\documentclass[a4paper,11pt]{article}
\usepackage{hyperref,tikz}
\usepackage{amsthm,amsmath,amssymb,amsfonts}
\usepackage[numbers, sort & compress]{natbib}

\def\pref#1{(\ref{#1})}

\makeatletter
\def\p@enumii{}
\makeatother

\newtheorem{thm}{Theorem}[section]
\newtheorem{lemma}[thm]{Lemma}
\newtheorem{proposition}[thm]{Proposition}
\newtheorem{corollary}[thm]{Corollary}

\theoremstyle{definition}

\theoremstyle{remark}
\newtheorem{remark}[thm]{Remark}
\numberwithin{equation}{section}

\renewcommand{\b}[1]{\overline{#1}}

\renewcommand{\b}{\big}

\newcommand{\G}{\mathcal{G}}
\newcommand{\inv}{\mathrm{Inv}}

\newcommand{\gr}{\mathrm{gr}}

\sloppypar

\begin{document}
\title{Involutory Cayley graphs of polynomial and power series rings over the ring of integers modulo $n$}

\author{Hamide Keshavarzi$^1$, Afshin Amini$^{2}$, Babak Amini$^3$\\ 
\it\small Department of Mathematics, College of Sciences, Shiraz University, \\
\it\small 71457-44776, Shiraz, Iran\\
\it\small $^1$E-mail: hamide2326@hafez.shirazu.ac.ir\\
\it\small $^2$E-mail: aamini@shirazu.ac.ir\\
\it\small $^3$E-mail: bamini@shirazu.ac.ir}
\date{}

\maketitle
\begin{abstract}
Let $R$ be a commutative ring with identity. The involutory Cayley graph $\G(R)$
of $R$ is defined as the graph whose vertex set is the set of elements of $R$, where two
vertices $a$ and $b$ are adjacent exactly when $(a-b)^2=1$. This paper investigates the properties of involutory Cayley graphs associated with polynomial and power series rings over the ring of integers modulo $n$. 
\end{abstract}


\section{Introduction}

Throughout this paper, all rings are assumed to be commutative and have an
identity. In addition, all graphs are simple
and undirected. $R$ always denotes a ring and  by $\mathbb{Z}_n$, we mean the ring of integers modulo $n$.

A significant and extensively researched class of highly symmetric graphs with strong ties to algebraic structures are Cayley graphs. Their analogues over rings have also attracted more attention recently, despite the fact that they were first defined in the context of arbitrary groups, where they have been thoroughly studied. In contrast to their group-theoretic counterparts, Cayley graphs frequently display rich algebraic and combinatorial properties when built over rings. Specifically, Cayley graphs over rings like $\mathbb{Z}_n$, polynomial rings, and matrix rings have been examined in connection with zero-divisor graphs, unit graphs, and other algebraic graph constructions (see, for example, \cite{akhtar,inv,inv2,qunit1,qunit3,unit1,unit2,unit3,unit4,unit5,unit6,zunit1,zunit2}).

Recall that an element $u$ of a ring is said to be \textit{involutory} or an \textit{involution} when $u=u^{-1}$. In this context, we consider the \textit{involutory Cayley graph} $\G(R)$ of a ring $R$, defined as the Cayley graph of the abelian group $(R,+)$ with respect to the set $\inv(R)$ of involutory elements of $R$. That is, the vertices of $\G(R)$ are the elements of $R$, and two vertices $a$ and $b$ are adjacent (denoted $a\sim b$) if and only if $(a-b)^2=1$. Note that in $\G(R)$, each vertex $a$ is adjacent to $a+u$ for all $u\in \inv(R)$, which implies that $\G(R)$ is a $|\inv(R)|$-regular graph with no isolated vertices. 

In \cite{inv}, we studied the involutory Cayley graphs of finite commutative rings. Let $R$ be a finite commutative ring. It was shown that $\G(R)$ is a complete graph if and only if  $R\simeq \mathbb{Z}_2$ or $\mathbb{Z}_3$. Moreover, $\G(R)$ is always a $2^t$-regular graph for some nonnegative integer $t$.  It was also proven that  $\G(R)$ is a cycle graph if and only if $R$ is isomorphic to one of the following rings:  $\mathbb{Z}_2[x]/\langle x^2 \rangle$, $\mathbb{Z}_4$, $\mathbb{Z}_{p^k}$,  or  $\mathbb{Z}_{2p^k}$,
where $p$ is an odd prime number and $k$ is a positive integer. Furthermore, $\G(R)$ is self-complementary if and only if $R$ is isomorphic to either  $\mathbb{Z}_5$ or $\mathbb{Z}_3\times \mathbb{Z}_3$. 
In addition, $\G(R)$ is bipartite if and only if $|R|$ is even.

Besides these results, the chromatic number, edge chromatic number, clique number, independence number and girth of $\G(R)$ were completely determined.  We also classified all finite commutative rings $R$ for which $\G(R)$ is planar. Extending our investigation, in \cite{inv2}, we further classified all finite commutative rings $R$ whose involutory Cayley graphs are toroidal.

We now recall the notion of the direct product of graphs. For two graphs $G_1$ and $G_2$, the \textit{direct product} (also known as the \textit{conjunction product} or \textit{Kronecker product}), denoted by $G_1\otimes G_2$, is defined as the graph with vertex set  $V(G_1)\times V(G_2)$, where two vertices $(g_1,g_2)$ and $(h_1,h_2)$ are adjacent in $G_1\otimes G_2$ if and only if  $g_i$ and $h_i$ are adjacent in $G_i$ for each $i=1,2$.

Let $R$ be a commutative ring with a decomposition $R\simeq R_1\times\cdots\times R_t$. Then,  the involutory Cayley graph of $R$ decomposes as $$\G(R)\simeq\G(R_1)\otimes \cdots\otimes \G(R_t).$$ To see this, observe  that for elements $x=(x_1,\ldots,x_t)$ and $y=(y_1,\ldots,y_t)$  in $R$, the difference $x-y$ is involutory in $R$ if and only if $x_i-y_i$  is involutory in $R_i$ for each $i$. Moreover, if each $\G(R_i)$ is $k_i$-regular, then $\G(R)$ is $k$-regular with $k=k_1\cdots k_t$.

For the ring $\mathbb{Z}_n$, some essential results from \cite{inv} are summarized in the following lemma.

\begin{lemma}\label{1.1}
For any positive integer $n\geq 2$, the following statements hold.
\begin{enumerate}
\item \cite[Theorem 3.1]{inv}
$\G(\mathbb{Z}_n)$ is bipartite if and only if $n$ is even;

\item \cite[Theorem 6.8]{inv} $\G(\mathbb{Z}_n)$ is planar if and only if $n=2$, $4$, $p^k$ or $2p^k$ for some odd prime $p$ and positive integer $k$;

\item \cite[Theorem 4.1]{inv} The clique number of $\G(\mathbb{Z}_n)$ is $3$ if $n=3$, otherwise it is $2$;

\item \cite[Corollary 3.2]{inv} The chromatic number of $\G(\mathbb{Z}_n)$ is $2$ if $n$ is even, otherwise it is $3$;

\item \cite[Theorem 5.4]{inv} For the girth of $\G(\mathbb{Z}_n)$, we have
\[   
\gr\b(\G(\mathbb{Z}_n)\b) = 
\begin{cases}
	\infty &\qquad n=2,\\
	n & \qquad n=p^k \, \text{or}\, \,  2p^k\, \text{for some odd prime} \, \, p,\\ 
	4 & \qquad  \text{otherwise}. \\ 
\end{cases}
\]
\end{enumerate}
\end{lemma}

This paper is organized as follows. In Section 2, we study the involutory elements of the polynomial and power series rings $\mathbb{Z}_n[x]$ and $\mathbb{Z}_n[[x]]$. As applications, in Section 3, we show that $\G(\mathbb{Z}_n[x])$ is always disconnected and determine exactly when this graph is bipartite or planar.  We also explore important graph parameters such as its independence number, clique number, chromatic number, and girth. Lastly, we point out that all these results carry over to $\G(\mathbb{Z}_n[[x]])$ as well.


\section{Involutory elements of $\mathbb{Z}_n[x]$ and $\mathbb{Z}_n[[x]]$}

In this section, we study the involutory elements of $\mathbb{Z}_n[x]$ and $\mathbb{Z}_n[[x]]$. We begin by exploring the involutions in $\mathbb{Z}_n$.

\begin{lemma}\label{invloc}
The following statements hold.
\begin{enumerate}
\item \label{invloc1}
$\inv(\mathbb{Z}_2)=\{1\}$ and $\inv(\mathbb{Z}_4)=\{\pm 1\}$;
\item \label{invloc2} $\inv(\mathbb{Z}_{2^k})=\{1,2^{k-1}-1,2^{k-1}+1,2^k-1\}$ for $k\geq 3$;
\item \label{invloc3}
$\inv(\mathbb{Z}_{p^k})=\{\pm 1\}$ for odd prime number $p$ and positive integer $k$.
\end{enumerate}
\end{lemma}
\begin{proof}
\pref{invloc1} is clear. 

\pref{invloc2} Assume that $u\in \inv(\mathbb{Z}_{2^k})$. This means that $2^k\mid (u-1)(u+1)$, implying that one of $u\pm 1$ must be of the form $4t$ and the other $4t+2$ for some integer $t$. This implies that $2^{k-3}\mid t$, and hence, $4t$ is either equal to $2^{k-1}$ or $2^{k}$ modulo $2^k$. From this, the result follows.

\pref{invloc3} Suppose that $u\in \inv(\mathbb{Z}_{p^k})$. Then $p^k \mid (u-1)(u+1)$, which implies that $p$ divides exactly one of the $u\pm 1$ since $p$ is an odd prime. Consequently, $p^k$ divides exactly one of the $u\pm 1$. From this, the result follows.
\end{proof}

Utilizing the above lemma, we can determine the number of involutions in $\mathbb{Z}_n$.

\begin{lemma}
Suppose that 
$n=2^{r_0}\times p_1^{r_1}\times \cdots \times p_t^{r_t}$, where $r_0$ is a nonnegative integer, $r_i$ are positive integers  for $i\geq 1$ and $p_i$ are distinct odd prime numbers. Then
$$|\inv(\mathbb{Z}_n)|=
\begin{cases}
	2^t &\qquad \quad r_0 = 0 \,\, \text{or}\, \,1 ,\\
	2^{t+1} & \qquad \quad r_0=2 , \\
	2^{t+2} & \qquad \quad r_0\geq 3 . 
\end{cases}$$
\end{lemma}
\begin{proof}
Note that for any two rings $R$ and $S$, we have
$$\inv(R\times S)= \inv(R)\times \inv(S).$$
Combining this with Lemma \ref{invloc}, the result follows.
\end{proof}

We now turn our attention to the involutions of polynomial rings.

\begin{proposition}\label{zp^k}
For any odd prime number $p$ and positive integer $k$, we have
$$\inv(\mathbb{Z}_{p^k}[x])=\{\pm 1\}.$$
\end{proposition}
\begin{proof}
Taking Lemma \ref{invloc} into account, we know that 
$\inv(\mathbb{Z}_{p^k})=\{\pm 1\}$. Suppose that 
$f=\sum_{i=0}^{t}a_ix^i\in \inv(\mathbb{Z}_{p^k}[x])$. Since $f^2=1$, we have $a_0^2=1$, which implies $a_0\in \{\pm 1\}$.

We now prove by induction on $i$ that $a_i=0$ for all $i\geq 1$. By considering the coefficient of $x$ on both sides of the equation $f^2=1$, we obtain $2a_0a_1=0$. Since $2$ is a unit in $\mathbb{Z}_{p^k}$ (as $p$ is an odd prime) and $a_0\in \{\pm 1\}$, it follows that $a_1=0$. Now, assume as the induction hypothesis that $a_1=a_2=\cdots = a_{i-1}=0$. Looking at the coefficient of  $x^i$ in the expansion of $f^2=1$, we obtain
$$a_0a_i+a_1a_{i-1}+\cdots+a_{i-1}a_1+a_ia_0=0.$$
By the induction hypothesis, all terms except $a_0a_i$ and $a_ia_0$ vanish, so we get $2a_0a_i=0$. Again, since $2$ is a unit and $a_0\in \{\pm 1\}$,  we conclude that $a_i=0$, completing the induction. Hence, the result follows.
\end{proof}

\begin{proposition}\label{z2^k}
For any integer $k\geq 2$, 
$$\inv(\mathbb{Z}_{2^k}[x])=\left\{a_0+2^{k-1}f\mid a_0\in \inv(\mathbb{Z}_{2^k}) \,\,\text{and} \,\, f\in \mathbb{Z}_{2^k}[x]\right\}.$$
In addition, we have
$\inv(\mathbb{Z}_{2}[x])=\{1\}.$
\end{proposition}
\begin{proof}
It is straightforward to verify that  $\inv(\mathbb{Z}_{2}[x])=\{1\}$.

Now, let $f= \sum_{i=0}^{t}a_ix^i\in \inv(\mathbb{Z}_{2^k}[x])$, where $k\geq 2$. Since $f^2=1$, it follows that $a_0\in \inv(\mathbb{Z}_{2^k})$. We prove by induction on $i$ that $a_i\in \{0,2^{k-1}\}$ for all $i\geq 1$. Considering the coefficient of $x$ in $f^2=1$, we obtain $2a_0a_1=0$. As $a_0$ is a unit, it follows that $2a_1=0$, and hence $a_1\in \{0,2^{k-1}\}$.

Now, assume as the induction hypothesis that $a_1,a_2,\cdots,a_{i-1}\in \{0,2^{k-1}\}$. Comparing the coefficient of $x^i$ on both sides of the equation $f^2=1$, we have
$$a_0a_i+a_1a_{i-1}+\cdots+a_{i-1}a_1+a_ia_0=0.$$
By the induction hypothesis, we deduce that $2a_0a_i=0$, and thus $2a_i=0$. Hence $a_i\in \{0,2^{k-1}\}$, as desired.
\end{proof}

The following theorem determines the number of involutions in polynomial rings over $\mathbb{Z}_n$. Recall that a graph is said to be \textit{connected} if there exists a path between any two of its vertices; otherwise, it is said to be \textit{disconnected}. A maximal connected subgraph of a graph is called a \textit{connected component}.

\begin{thm}\label{main}
Suppose that 
$n=2^{r_0}\times p_1^{r_1}\times \cdots \times p_t^{r_t}$, where $r_0$ is a nonnegative integer, $r_i$ are positive integers  for $i\geq 1$ and $p_i$ are distinct odd prime numbers. The following statements hold.
\begin{enumerate}
\item 
If $r_0\in \{0,1\}$, then $|\inv(\mathbb{Z}_n[x])|=2^t$. Moreover, the graph $\G(\mathbb{Z}_n[x])$ is an infinite disjoint union of subgraphs isomorphic to $\G(\mathbb{Z}_n)$;

\item 
If $r_0\geq 2$, then $\inv(\mathbb{Z}_n[x])$ is infinite.
\end{enumerate}
\end{thm}
\begin{proof}
Note that for any two rings $R$ and $S$, we have
$$(R\times S)[x] \simeq R[x]\times S[x].$$
So, 
$$\mathbb{Z}_n[x]\simeq \mathbb{Z}_{2^{r_0}}[x]\times \mathbb{Z}_{p_1^{r_1}}[x]\times \cdots \times \mathbb{Z}_{p_t^{r_t}}[x],$$
which implies that
$$\inv(\mathbb{Z}_n[x])=\inv(\mathbb{Z}_{2^{r_0}}[x])\times \inv(\mathbb{Z}_{p_1^{r_1}}[x])\times \cdots \times \inv(\mathbb{Z}_{p_t^{r_t}}[x]).$$
Now, by Propositions \ref{zp^k} and \ref{z2^k}, the result follows.

Furthermore, for the last statement of part (i), note that if $r_0\in \{0,1\}$, then  $\G(\mathbb{Z}_n)$ is a $2^t$-regular subgraph of the $2^t$-regular graph $\G(\mathbb{Z}_n[x])$. Hence, $\G(\mathbb{Z}_n)$ is a connected component of $\G(\mathbb{Z}_n[x])$. Since all connected components of a Cayley graph are isomorphic, the result follows.
\end{proof}

As an immediate application of the above theorem, we obtain the following corollary.

\begin{corollary}\label{pk2pk}
Let $n\geq 2$ be a positive integer. Then $|\inv(\mathbb{Z}_n[x])|=2$ if and only if $n=p^k$ or $n=2p^k$ for some odd prime number $p$ and positive integer $k$. 

Consequently, the graphs $\G(\mathbb{Z}_{p^k}[x])$ and $\G(\mathbb{Z}_{2p^k}[x])$ consist precisely of  infinitely many disjoint cycles of length $p^k$ and $2p^k$, respectively.
\end{corollary}
\begin{proof}
We only have to prove the "Consequently" part. Note that any $f\in \mathbb{Z}_n[x]$, lies in the cycle 
$$f \sim (f+1) \sim (f+2) \sim \cdots \sim (f+n-1) \sim f$$
in $\G(\mathbb{Z}_n[x])$. So, if $n=p^k$ or $n=2p^k$ for some odd prime number $p$ and positive intger $k$, then $\G(\mathbb{Z}_n[x])$ is a $2$-regular graph, making it a disjoint union of cycles of length $n$ and hence, the result follows.
\end{proof}

\begin{remark}
It is not difficult to verify that all the results of this section remain valid when $\mathbb{Z}_n[x]$ is replaced by $\mathbb{Z}_n[[x]]$.
\end{remark}


\section{Involutory Cayley graphs of $\mathbb{Z}_n[x]$ and $\mathbb{Z}_n[[x]]$}

In this section, we study the involutory Cayley graphs of  $\mathbb{Z}_n[x]$. As in Section 2, we conclude by noting that all results also hold for $\mathbb{Z}_n[[x]]$.

We begin  by examining the connectedness of  $\G(\mathbb{Z}_n[x])$.  It is noteworthy that $\G(\mathbb{Z}_n)$ is always connected,  as it contains the cycle
$$0\sim 1 \sim 2 \sim \cdots \sim (n-1) \sim 0,$$
which connects all the elements. However, the situation is quite the opposite for polynomial rings.

\begin{thm}\label{com}
For any $n\geq 2$, the involutory Cayley graph of $\mathbb{Z}_n[x]$ has infinitely many connected components.
\end{thm}
\begin{proof}
Suppose that 
$n=2^{r_0}\times p_1^{r_1}\times \cdots \times p_t^{r_t}$, where $r_i$ are nonnegative integers and $p_i$ are distinct odd prime numbers. Then, we have the decomposition
$$\G(\mathbb{Z}_n[x])\simeq\G(\mathbb{Z}_{2^{r_0}}[x])\otimes \G(\mathbb{Z}_{p_1^{r_1}}[x])\otimes \cdots \otimes \G(\mathbb{Z}_{p_t^{r_t}}[x]).$$
 In the following, we show that each factor in this direct product has infinitely many connected components.

According to Corollary \ref{pk2pk}, the graph $\G(\mathbb{Z}_{p_i^{r_i}}[x])$  is a disjoint union of infinitely many cycles of length $p_i^{r_i}$. Also, $\G(\mathbb{Z}_2[x])$ is $1$-regular by Theorem \ref{main}, having infinite connected components.

Now, consider the graph $\G(\mathbb{Z}_{2^{r_0}}[x])$, where $r_0\geq 2$. Suppose  $f=\sum_{i=0}^{s}a_ix^i$ and $g=\sum_{i=0}^{m}b_ix^i$ are adjacent vertices in this graph. Then, by Proposition \ref{z2^k}, we have  $2^{r_0-1}\mid a_i - b_i$ in $\mathbb{Z}_{2^{r_0}}$ for all $i\geq 1$.  In particular, this implies that the elements of the set $\{1,x,x^2,\ldots\}$ lie in different connected components. Hence, the result follows.
\end{proof}

\begin{corollary}
For any $n\geq 2$, the involutory Cayley graph of $\mathbb{Z}_n[x]$ is disconnected.
\end{corollary}

In the following theorem, we characterize when $\G(\mathbb{Z}_n[x])$ is bipartite. Recall that a \textit{bipartite graph} is a graph whose vertex set can be partitioned into two disjoint subsets such that every edge connects a vertex in one subset to a vertex in the other.

\begin{thm}\label{bip}
For a positive integer $n\geq 2$, the following conditions are equivalent. 
\begin{enumerate}
\item The involutory Cayley graph of $\mathbb{Z}_n[x]$ is bipartite;

\item The involutory Cayley graph of $\mathbb{Z}_n$ is bipartite;

\item $n$ is an even integer.
\end{enumerate}
\end{thm}
\begin{proof}
If $n$ is an odd integer, then $\G(\mathbb{Z}_n[x])$ contains the cycle
$$0\sim 1 \sim \cdots \sim (n-1)\sim 0,$$
which has odd length. Therefore, $\G(\mathbb{Z}_n[x])$ cannot be bipartite. 

On the other hand, suppose $n$ is  even. Then for any $u\in \mathbb{Z}_n$ with $u^2=1$, we have $n\mid u^2-1$, which implies $u\not\in 2\mathbb{Z}_n$. In other words, the set $2\mathbb{Z}_n$ contains no involutory elements. Now, by setting 
$$A=\{f(x) \in \mathbb{Z}_n[x]\mid f(0)\in 2\mathbb{Z}_n\}$$
and 
$$B=\{f(x) \in \mathbb{Z}_n[x]\mid f(0) \not\in 2\mathbb{Z}_n\},$$
it is not hard to verify that $A\cup B$ forms a bipartition of the graph $\G(\mathbb{Z}_n[x])$. This, together with Lemma \ref{1.1}, completes the proof.
\end{proof}

Next, we determine the conditions under which the graph $\G(\mathbb{Z}_n[x])$ is planar. Recall that a graph is said to be \textit{planar} if it can be embedded in the plane such that no two edges intersect except at their endpoints.

\begin{thm}
For a positive integer $n\geq 2$, the following conditions are equivalent.
\begin{enumerate}
\item 
The involutory Cayley graph of $\mathbb{Z}_n[x]$ is planar;

\item 
The involutory Cayley graph of $\mathbb{Z}_n$ is planar and $n\ne 4$;

\item 
$n=2$, $p^k$ or $2p^k$ for some odd prime number $p$ and positive integer $k$.
\end{enumerate}
\end{thm}
\begin{proof}
First, suppose that $\G(\mathbb{Z}_n[x])$ is planar. Since $\G(\mathbb{Z}_n)$ is a subgraph of $\G(\mathbb{Z}_n[x])$, it follows that $\G(\mathbb{Z}_n)$ must also be planar. By Lemma \ref{1.1}, this implies that $n=2$, $4$, $p^k$ or $2p^k$ for some odd prime number $p$ and positive integer $k$. 

However, $\G(\mathbb{Z}_4[x])$ is not planar. To see this, note that any element of $A=\{0,2,2x\}$ is adjacent to every element of $1+A=\{1,3,1+2x\}$,  implying that  the subgraph induced by  $A\cup (1+A)$ is the complete bipartite graph $K_{3,3}$, which is nonplanar.

On the other hand, by Proposition \ref{z2^k} and Corollary \ref{pk2pk}, we know that $\G(\mathbb{Z}_2[x])$ is $1$-regular, while  $\G(\mathbb{Z}_{p^k}[x])$ and $\G(\mathbb{Z}_{2p^k}[x])$ are $2$-regular. These graphs are clearly planar. This completes the proof.
\end{proof}

An \textit{independent set} in a graph is a set of vertices no two of which are adjacent. The \textit{independence number} of a graph is the cardinality of a maximum independent set. According to \cite[Theorem 4.3]{inv}, the independence number of $\G(\mathbb{Z}_n)$ is equal to $n/2$ when  $n$ is even. Moreover, if $n=p_1^{r_1} \cdots  p_t^{r_t}$ is the prime decomposition of an odd integer $n$ with
$p_1^{r_1}\geq \cdots \geq p_t^{r_t}$, then the independence number of $\G(\mathbb{Z}_n)$ is given by
$(p_1^{r_1}-1)n/2p_1^{r_1}$. In contrast, we have the following theorem for $\G(\mathbb{Z}_n[x])$.

\begin{thm}\label{alpha}
For any positive integer $n \geq 2$, the independence number of $\G(\mathbb{Z}_n[x])$ is infinite (in fact, it is equal to $\aleph_0$).
\end{thm}
\begin{proof}
It is not difficult to see that $\{1,x,x^2,\ldots\}$ is an independent set (also, see Theorem \ref{com}). From this, the result follows.
\end{proof}

In the next theorem, we determine the chromatic number of $\G(\mathbb{Z}_n[x])$. Recall that the \textit{chromatic number} $\chi(G)$ of a graph $G$ is the minimum number of colors needed to color vertices of $G$ such that no two adjacent vertices receive the same color.

\begin{thm}\label{chi}
Let $n\geq 2$ be a positive integer. Then
\[   
\chi\b(\G(\mathbb{Z}_n[x])\b)=\chi\b(\G(\mathbb{Z}_n)\b) = 
\begin{cases}
	2 &\qquad n  \text{  is even,}\\
	3 & \qquad n   \text{ is odd.} \\ 
\end{cases}
\]
\end{thm}
\begin{proof}
If $n$ is an even integer, by Theorem \ref{bip}, $\G(\mathbb{Z}_n[x])$ is a bipartite graph and hence $\chi\b(\G(\mathbb{Z}_n[x])\b)=2$.

Now, suppose $n$ is odd. By Theorem \ref{main}, the graph $\G(\mathbb{Z}_n[x])$ is a disjoint union of isomorphic copies $\G(\mathbb{Z}_n)$. Consequently, 
$$\chi\b(\G(\mathbb{Z}_n[x])\b)=\chi\b(\G(\mathbb{Z}_n)\b).$$
Taking Lemma \ref{1.1} into account, the result follows.
\end{proof}

As a consequence, we can determine  the clique number of $\G(\mathbb{Z}_n[x])$. Recall that the \textit{clique number} $\omega(G)$ of a graph $G$ is the cardinality of the largest set of vertices in which every pair is adjacent.

\begin{thm}\label{w}
	Let $n\geq 2$ be a positive integer. Then
	\[   
	\omega\b(\G(\mathbb{Z}_n[x])\b)=\omega\b(\G(\mathbb{Z}_n)\b) = 
	\begin{cases}
		2 &\qquad n  \ne 3,\\
		3 & \qquad n   =3 . \\ 
	\end{cases}
	\]
\end{thm}
\begin{proof}
If $n$ is even, then by Theorem \ref{bip}, the graph $\G(\mathbb{Z}_n[x])$ is bipartite, implying that its clique number is equal to $2$. Moreover, if $n$ is odd, by Theorem \ref{main}, the graph $\G(\mathbb{Z}_n[x])$ is a disjoint union of isomorphic copies of $\G(\mathbb{Z}_n)$, so 
$$	\omega\b(\G(\mathbb{Z}_n[x])\b)=\omega\b(\G(\mathbb{Z}_n)\b).$$
Now, Lemma \ref{1.1} completes the proof.
\end{proof}

A graph is said to be \textit{self-complementary} if it is isomorphic to its complement. According to \cite[Theorem 2.10]{inv}, the graph $\G(\mathbb{Z}_n)$ is self-complementary if and only if $n=5$. However, $\G(\mathbb{Z}_n[x])$ is never self-complementary. This follows from Theorems \ref{alpha} and \ref{w}, which show that the clique number and the independence number of $\G(\mathbb{Z}_n[x])$ are never equal.

Finally, we determine the girth of $\G(\mathbb{Z}_n[x])$. The \textit{girth} $\gr(G)$ of a graph $G$  is defined as the length of its shortest cycle; if $G$ contains no cycles, then $\gr(G)=\infty$.

\begin{thm}
Let $n\geq 2$ be a positive integer. Then
\[   
\gr\b(\G(\mathbb{Z}_n[x])\b) =  \gr\b(\G(\mathbb{Z}_n)\b)=
\begin{cases}
	\infty &\qquad n=2,\\
	n & \qquad n=p^k \, \text{or}\, \,  2p^k\, \text{for some odd prime} \, \, p,\\ 
	4 & \qquad  \text{otherwise}. \\ 
\end{cases}
\]
\end{thm}
\begin{proof}
If $n=2$, $n=p^k$ or $n=2p^k$ for some odd prime $p$, then the result follows directly from Proposition \ref{z2^k} and Corollary \ref{pk2pk}. So, assume that none of these cases occur.  Since 
$$\gr\b(\G(\mathbb{Z}_n[x])\b)\leq \gr\b(\G(\mathbb{Z}_n)\b),$$
by Lemma \ref{1.1}, we infer that $\gr(\G(\mathbb{Z}_n[x]))\leq 4$. On the other hand, by Theorem \ref{w}, we know that $\G(\mathbb{Z}_n[x])$ does not contain a cycle  of length $3$. Therefore, we conclude that 
$\gr(\G(\mathbb{Z}_n[x]))= 4,$
as claimed.
\end{proof}

\begin{remark}
As in Section 2,  all the results of this section also hold when $\mathbb{Z}_n[x]$ is replaced by $\mathbb{Z}_n[[x]]$.
\end{remark}


\end{document}